\documentclass[11pt]{amsart}

\usepackage[arrow, matrix, curve]{xy}
\usepackage{amscd}
\usepackage{amssymb}
\usepackage{amsthm}
\usepackage{amsmath}
\usepackage{mathrsfs}

\usepackage{stmaryrd}

\newtheorem{theorem}{Theorem}[section]

\theoremstyle{plain}
\newtheorem{definition}[theorem]{Definition}

\newtheorem{prop}[theorem]{Proposition}
\newtheorem{cor}[theorem]{Corollary}
\newtheorem{lemma}[theorem]{Lemma}
\newtheorem{thm}[theorem]{Theorem}
\newtheorem{notation}[theorem]{Notation}

\theoremstyle{definition}

\theoremstyle{remark}
\newtheorem{remark}[theorem]{Remark}

\newcommand{\OC}{\mathcal{O}}

\newcommand{\PP}{\mathcal{P}}
\newcommand{\W}{\mathcal{W}}
\newcommand{\wW}{\widehat{\mathcal{W}}}

\newcommand{\N}{\mathbb{N}}
\newcommand{\Z}{\mathbb{Z}}

\newcommand{\C}{\mathbb{C}}
\newcommand{\K}{\mathbb{K}}

\newcommand{\orho}{\overline{\rho}}
\newcommand{\wb}{\widehat{\mathfrak{b}}}
\newcommand{\wh}{\widehat{\mathfrak{h}}}
\newcommand{\wg}{\widehat{\mathfrak{g}}}
\newcommand{\g}{\mathfrak{g}}
\newcommand{\lb}{\mathfrak{b}}
\newcommand{\h}{\mathfrak{h}}

\newcommand{\wS}{\widehat{S}}
\newcommand{\wR}{\widehat{R}}

\newcommand{\ch}{\mathrm{ch}}
\newcommand{\hc}{\widehat{\mathfrak{h}}_{\mathrm{crit}}^*}

\newcommand{\al}{{\alpha}}
\newcommand{\lam}{{\lambda}}
\newcommand{\Lam}{{\Lambda}}
\newcommand{\del}{{\Delta}}
\newcommand{\rdel}{{\overline{\Delta}}}

\newcommand{\Hom}{\mathrm{Hom}}

\numberwithin{equation}{section}

\begin{document}

\title[Jantzen Sum formula]{Jantzen sum formula for restricted Verma modules over affine Kac-Moody algebras at the critical level}

\author{Johannes K\"ubel*}
\address{Department of Mathematics, University of Erlangen, Germany}
\curraddr{Cauerstr. 11, 91058 Erlangen, Germany}
\email{kuebel@mi.uni-erlangen.de}
\thanks{*supported by the DFG priority program 1388}


\date{\today}

\keywords{critical representations of affine Kac-Moody algebras, Jantzen filtration, category $\OC$}

\begin{abstract}
For a restricted Verma module of an affine Kac-Moody algebra at the critical level we describe the Jantzen filtration and calculate its character. This corresponds to the Jantzen sum formula of a baby Verma module over a modular Lie algebra. This also implies a new proof of the linkage principle which was already derived by Arakawa and Fiebig.
\end{abstract}

\maketitle

\section{Introduction}  

To a simple Lie algebra $\g$ with Cartan subalgebra $\h$ one associates an affine Lie algebra $\wg$ with Cartan subalgebra $\wh$. The root system $R$ of $\g$ can be embedded into the root system $\wR$ of $\wg$. In \cite{3} the authors introduce the category $\overline{\OC}_c$ of restricted representations of $\wg$ at the critical level. Denote by $\OC_c$ the direct summand of the usual highest weight category $\OC$ over $\wg$ which consists of modules with critical level. Then $\overline{\OC}_c$ is the subcategory of $\OC_c$ on which those elements of the Feigin-Frenkel center act by zero which are homogeneous of degree unequal to zero.\\
Conjecturally, the restricted category $\overline{\OC}_c$ should have the same structure as the representation category over a small quantum group or a modular Lie algebra described in \cite{1}. The standard modules in $\overline{\OC}_c$, which should correspond to baby Verma modules in the representation categories of \cite{1}, are the so called \textit{restricted Verma modules}. They are the maximal restricted quotients of the ordinary Verma modules.\\
Towards the description of $\overline{\OC}_c$, Arakawa and Fiebig confirmed the above conjecture in the subgeneric (cf. \cite{2}) and Frenkel and Gaitsgory in the generic case. In \cite{1}, chapter 6, and also in \cite{8} the authors compute a Jantzen sum formula for a baby Verma module $Z(\lam)$ which describes the characters of the Jantzen filtration as an alternating sum formula of certain characters of baby Verma modules of weight "lower" than $\lam$.\\
We deduce the analogous formula for restricted Verma modules at the critical level. To be more precise, let $\lam \in \wh^*$ be a weight of critical level. We introduce the Jantzen filtration 
$$\rdel(\lam)=\rdel(\lam)^0\supset \rdel(\lam)^1 \supset \rdel(\lam)^2 \supset ...$$
and deduce the formula
$$\sum_{\substack{i>0}} \mathrm{ch} \rdel(\lam)^i =  \sum_{\substack{\alpha \in R(\lam)^{+}}} \left(\sum_{\substack{i>0}} \left(\mathrm{ch} \rdel(\alpha \downarrow^{2i-1}\lam)- \mathrm{ch} \rdel(\alpha \downarrow^{2i}\lam)\right)\right)$$
Here $R(\lam)^+\subset R$ denotes the positive integral roots of the finite root system $R$. The notation $\al \downarrow^i\lam$ for $i\leq0$ is defined inductively by $\al \downarrow (\al \downarrow^{i-1}\lam)$ where $\al \downarrow \lam =s_\al \cdot \lam$ if $s_\al \cdot \lam \leq \lam$ and $\al \downarrow \lam =s_{-\al+\delta} \cdot \lam$ if $s_\al \cdot \lam > \lam$. Here $\delta \in \wR$ denotes the smallest positive imaginary root and $s_\al$, $s_{-\al+\delta}$ are the reflections corresponding to the real roots $\al, -\al +\delta$ of the affine Weyl group $\wW$ with its dot-action on $\wh^*$.\\
The strategy to proof the Jantzen sum formula is to deduce the subgeneric cases first and then put these together to get the general result in a very similar manner as in \cite{7}, chapters 5.6 and 5.7. To deduce the subgeneric case we use a result of \cite{2} which states that for $\lam\in \wh^*$ critical and subgeneric the maximal submodule of $\rdel(\lam)$ is isomorphic to the simple module $L(\al\downarrow \lam)$ with highest weight $\al\downarrow \lam$.\\
In \cite{3} the authors introduce projective objects in the restricted category $\overline{\OC}_c$ and a BGGH-reciprocity to deduce the linkage principle for restricted Verma modules conjectured by Feigin and Frenkel. It states that if the simple module with highest weight $\mu \in \wh^*$ appears as a subquotient in a Jordan-H\"older series of $\rdel(\lam)$, where $\lam \in \wh^*$ is critical, then $\mu \in \wW(\lam) \cdot \lam$. Here $\wW(\lam)$ denotes the integral affine Weyl group of the root system $R(\lam)$ corresponding to $\lam$ with its dot-action on $\wh^*$.\\
The linkage principle immediately follows from the Jantzen sum formula and thus gives an independent proof.

\section{Preliminaries}

In this chapter we shortly introduce the construction of an (untwisted) affine Kac-Moody algebra starting with a simple Lie algebra. We collect some facts about the root data, the Weyl group and the invariant bilinear form.

\subsection{Affine Kac-Moody algebras}

Let $\g$ be a simple Lie algebra with a Borel subalgebra $\lb$ and a Cartan subalgebra $\h$. We denote by $R$ the root system with positive roots $R^+$ and by $\Pi$ the simple roots. Moreover, denote by $\W$ the finite Weyl group and by $\kappa: \g \times \g \rightarrow \C$ the Killing form.\\
We take a non-split central extension $\widetilde{\g}$ of the loop algebra $\g \otimes_\C \C[t,t^{-1}]$. As a vector space $\widetilde{\g}$ is the direct sum $(\g \otimes_\C \C[t,t^{-1}]) \oplus \C c$, where $c$ is a central element.\\
Adding a derivation operator $d$ with the property $[d,\cdot] =t \frac{\partial}{\partial t}$, we get the affine Kac-Moody algebra $\wg$ associated to $\g$. As a vector space we have $\wg = (\g \otimes_\C \C[t,t^{-1}]) \oplus \C c \oplus \C d$ and the Lie bracket is given by
$$[c,\wg]= \{0\}$$
$$[d,x\otimes t^n] =nx \otimes t^n$$
$$[x\otimes t^n, y \otimes t^m] = [x,y] \otimes t^{m+n} + n \delta_{m, -n} \kappa(x,y)c$$
where $x,y \in \g$ and $n\in \Z$.
The Borel subalgebra of $\wg$ corresponding to $\lb \subset \g$ is given by
$$\wb = (\g \otimes_\C t\C[t] +\lb \otimes_{\C} \C[t]) \oplus \C c \oplus \C d$$
while the corresponding Cartan subalgebra of $\wg$ is given by 
$$\wh= \h \oplus \C c \oplus \C d.$$

\subsection{Affine roots, Weyl groups and bilinear forms}

For a vector space $V$ we denote by $\left\langle\cdot, \cdot \right\rangle : V^* \times V \rightarrow \C$ the natural pairing with its dual space. Denote by $\leq$ the usual ordering on $\wh^*$. The projection $\wh=\h \oplus \C c\oplus \C d  \rightarrow \h$ induces an embedding $\h^* \subset \wh^*$. By this embedding we can consider all finite roots as elements of $\wh^*$. We define two weights $\Lam_\circ, \delta \in \wh^*$ by the relations
$$\left\langle \delta , \h \oplus \C c \right\rangle = \{0\}$$
$$\left\langle \delta, d \right\rangle =1$$
$$\left\langle \Lam_\circ , \h \oplus \C d \right\rangle =\{0\}$$
$$\left\langle \Lam_\circ , c \right\rangle=1 $$

Then the roots of $\wg$ with respect to $\wh$ are given by $\widehat{R} =\widehat{R}^{\mathrm{re}} \cup \widehat{R}^{\mathrm{im}}$ where 

$$\widehat{R}^{\mathrm{re}}:= \{\al +n \delta\,|\,\al \in R\subset \widehat{R}, \, n \in \Z \}$$
are the \textit{real roots}, and 
$$\widehat{R}^{\mathrm{im}} :=\{n \delta \,|\, n \in \Z, n \neq 0\}$$
the \textit{imaginary roots}. The positive roots $\widehat{R}^+$, i.e., the roots of $\wb$ with respect to $\wh$, are then given by
$$\widehat{R}^+ =R^+ \cup \{\al + n \delta \,|\, \al \in R, n>0\} \cup \{n \delta|n>0\}$$
Denote by $\theta$ the highest root of $R$. Then the set of simple affine roots is given by 
$$\widehat{\Pi} = \Pi \cup \{-\theta +\delta\}\subset \widehat{R}^+$$

For a real root $\al \in \widehat{R}^\mathrm{re}$ we denote by $\al^\vee \in \wh$ its \textit{coroot} which is uniquely defined by the properties $\al^\vee \in [\wg_\al, \wg_{-\al}]$ and $\left\langle \al, \al^\vee \right\rangle=2$.\\
We denote by $\wW \subset \mathrm{Gl}(\wh^*)$ the \textit{affine Weyl group} of the root system $\widehat{R}$, which is the subgroup generated by the reflections $s_\al : \wh^* \rightarrow \wh^*$, $\lam \mapsto \lam - \left\langle \lam, \al^\vee \right\rangle \al$ where $\al \in \widehat{R}^{\mathrm{re}}$ is a real root. We can identify the finite Weyl group $\W$ with the subgroup of $\wW$ generated by the reflections $s_\al$ corresponding to finite roots $\al \in R$. Then $\W$ stabilizes the subspace $\h^* \subset \wh^*$.\\
Let $\orho:= \frac{1}{2}\sum_{\al \in R^+} \al$ be the halfsum of positive finite roots. We then set 
$$\rho:= \orho +h^\vee \Lam_\circ$$
where $h^\vee$ is the dual Coxeter number of $\g$. Then $\left\langle \rho, \al^\vee\right\rangle \neq 0$ for all $\al \in \widehat{R}^\mathrm{re}$ and $\left\langle \rho, c\right\rangle \neq 0$ as well. We can now define the $\rho$-shifted dot-action of $\wW$ on $\wh^*$ by
$$w \cdot \lam:= w(\lam+\rho)-\rho$$
where $w\in \wW$ and $\lam \in \wh^*$.\\
The Killing form $\kappa$ on the simple Lie algebra $\g$ extends to a bilinear form $(\cdot|\cdot):\wg\times \wg \rightarrow \C$ which is given by the following equations
$$(x\otimes t^n| y\otimes t^m)=\delta_{n,-m}\kappa(x,y),$$
$$(c|\g \otimes_\C \C[t,t^{-1}] \oplus \C c)=\{0\},$$
$$(d| \g \otimes_\C \C[t,t^{-1}] \oplus \C d)=\{0\},$$
$$(c|d)=1,$$
for $x,y \in \g$, $m,n \in \Z$. It is non-degenerate, symmetric and invariant, i.e., $([x,y]|z)=(x|[y,z])$ for all $x,y,z \in \wg$. Furthermore, it induces a non-degenerate bilinear form on the affine Cartan subalgebra and thus an isomorphism $\nu:\wh \stackrel{\sim}{\rightarrow} \wh^*$ which coincides with the isomorphism $\h \stackrel{\sim}{\rightarrow} \h^*$ induced by the Killing form, when restricted to the finite Cartan subalgebra, and which sends $c$ to $\delta$ and $d$ to $\Lam_\circ$. So the induced form on $\wh^*$ is given by
$$(\al| \beta) =\kappa(\al,\beta)$$
$$(\Lam_\circ| \h^* \oplus \C \Lam_\circ)=\{0\},$$
$$(\delta| \h^* \oplus \C \delta)=\{0\},$$
$$(\Lam_\circ| \delta)=1,$$
for $\al, \beta \in \h^*$ and $\kappa$ the induced Killing form on $\h^*$. The induced form is invariant under the linear action of the affine Weyl group, i.e.,
$$(w(\lam)|w(\mu))=(\lam|\mu)$$
for $\lam, \mu \in \wh^*$, $w \in\wW$.

\section{Verma modules}

For a Lie algebra $\mathfrak{a}$ we denote by $U(\mathfrak{a})$ its universal enveloping algebra. For $\lam \in \wh^*$ let $\C_\lam$ be the one dimensional representation of $U(\wb)$ on which $\wh$ acts by the character $\lam$ and $[\wb,\wb]$ by zero. Then the \textit{Verma module with highest weight $\lam$} is defined by
$$\del(\lam):= U(\wg)\otimes_{U(\wb)} \C_\lam$$
It has a unique simple subquotient which we denote by $L(\lam)$ and both modules are highest weight modules with highest weight $\lam$.

\subsection{Deformed Verma modules}

Denote by $\wS:=S(\wh)=U(\wh)$ the symmetric algebra over the vector space $\wh$ and by $S=S(\h)$ the symmetric algebra over the vector space $\h$. Then the projection $\wh \rightarrow \h$ induces a homomorphism $\wS \rightarrow S$ and equips $S$ with an $\wS$-algebra structure. We call a commutative, unital, noetherian $\wS$-algebra with structure morphism $\tau:\wS \rightarrow A$ a \textit{deformation algebra}.\\
For a Lie algebra $\mathfrak{a}$ we set $\mathfrak{a}_A:= \mathfrak{a}\otimes_\C A$. Then we can identify $(\wh_A)^*=\Hom_A(\wh_A,A)$ with $\wh^* \otimes_\C A$ and any weight $\lam \in \wh^*$ induces a weight $\lam \otimes 1 \in \wh_A^*$ which we simply denote by $\lam$ again. In this way, the composition $\wh \hookrightarrow \wS \stackrel{\tau}{\rightarrow} A$ induces the \textit{canonical weight} $\tau \in \wh_A^*$.\\
For $\lam \in \wh^*$ let $A_\lam$ be the $\wb_A$-module that is $A$ as an $A$-module and on which $\wh$ acts via the character $\lam+\tau$ and $[\wb,\wb]$ by zero. We then define the \textit{deformed Verma module with highest weight} $\lam$ by
$$\del_A(\lam +\tau) := U(\wg_A) \otimes_{U(\wb_A)} A_\lam$$
For an $\wh_A$-module $M$ and $\lam \in \wh^*$ we define the \textit{deformed weight space} of $\lam$ by
$$M_\lam := \{m\in M\,|\, Hm =(\lam + \tau)(H) m \,\mathrm{for}\,\mathrm{all} \, H \in \wh_A \}$$
Then the deformed Verma module $\del_A(\lam + \tau)$ decomposes into the direct sum of its weight spaces $\del_A(\lam +\tau)_\mu$ with $\mu \in \wh^*$, such that $\del_A(\lam +\tau)_\mu \neq 0$ implies $\mu \leq \lam$.\\
If $A \rightarrow A'$ is a homomorphism of deformation algebras with structure maps $\tau:\wS \rightarrow A$ and $\tau':\wS \rightarrow A\rightarrow A'$, then
$$\del_A (\lam +\tau) \otimes_A A' \cong \del_{A'} (\lam + \tau').$$
Note that for $\tau: \wS \rightarrow \C$ the surjection on the quotient $\C \cong \wS/\wh \wS$ we have $\del_\C(\lam +\tau) \cong \del(\lam)$.

\subsection{Characters}

Let $\Z[\wh^*]=\bigoplus_{\lam \in \wh^*} \Z e^\lam$ be the group algebra of $\wh^*$. We define a certain completion by $\widehat{\Z[\wh^*]} \subset \prod_{\lam \in \wh^*} \Z e^\lam$ to be the subgroup of elements $(c_\lam)$ with the property that there exists a finite subset $\{\mu_1,...,\mu_n\} \subset \wh^*$ such that $c_\lam \neq 0$ implies $\lam \leq \mu_i$ for at least one $i$. Let $M\in \wg$-mod be semi-simple over $\wh$ with the properties that each weight space $M_\lam$ is finite dimensional and that there exists $\mu_1,...,\mu_n \in \wh^*$ such that $M_\lam  \neq 0$ implies $\lam \leq \mu_i$ for at least one $i$. Then we define the \textit{character} of $M$ as an element in $\widehat{\Z[\wh^*]}$ which is given by the formal sum
$$\ch M := \sum_{\lam \in \wh^*} (\mathrm{dim}_\C M_\lam) e^\lam$$
We define the \textit{generalized Kostant partition function} $\PP : \Z\widehat{R}^+ \rightarrow \N$ by
$$\PP(\nu):=\begin{cases}  \mathrm{dim}_\C \del(0)_\nu , \, \, \mathrm{if}\, \, \nu \in \N\widehat{R}^+ \\ 0, \, \, \mathrm{otherwise} \end{cases}$$

The name partition function comes from the combinatorial description of the formula
$$\ch \del(\lam)= \prod_{\al \in \widehat{R}^+}(1+e^{-\al} + e^{-2\al}+...)^{\mathrm{dim}\wg_\al}$$

\section{Restricted Verma modules}

\subsection{An equivalence relation}\label{KK}

For a deformation algebra $A$ with canonical weight $\tau: \wh_A \rightarrow A$, we extend the bilinear form $(\cdot|\cdot):\wh^* \times \wh^* \rightarrow \C$ to an $A$-linear continuation $(\cdot| \cdot)_A: \wh_A^* \times \wh_A^* \rightarrow A$. \\
If $A = \K$ is a field, we denote by $\preceq_\K$ the partial ordering on $\wh^*$ which is generated by $\nu \preceq_\K \mu$ if there exists an $n \in \N$ and an $\al \in \widehat{R}^+$, such that $2(\lam +\rho+\tau| \al)_\K=n(\al| \al)_\K$ and $\nu =\lam -n \al$. Then $\preceq_\K$ is a refinement of the usual ordering $\leq$ on $\wh^*$. We denote by $\sim_\K$ the equivalence relation on $\wh^*$ which is generated by $\preceq_\K$.\\
Let $L_\K(\lam+\tau)$ be the unique simple quotient of $\del_\K(\lam +\tau)$ and denote by $[\del_\K(\lam+\tau):L_\K(\mu+\tau)]$ the number of subquotients of a Jordan-H\"older series of $\del_\K(\lam+\tau)$ which are isomorphic to $L_\K(\mu +\tau)$.

\begin{thm}[\cite{6}, Theorem 2]\label{kk}
We have $[\del_\K(\lam +\tau):L_\K(\mu +\tau)]\neq 0$ if and only if $\mu \preceq_\K \lam$.
\end{thm}

For $\lam \in \wh^*$ we define the \textit{integral roots} (with respect to $\lam$ and $A$) by
$$\widehat{R}_A(\lam):= \{\al \in \widehat{R} \, |\, 2(\lam +\rho +\tau| \al)_A \in \Z(\al|\al)_A\}$$
and the corresponding \textit{integral Weyl group} by
$$\wW_A(\lam):=\left\langle s_\al \,|\, \al \in \widehat{R}_A(\lam)\cap \widehat{R}^\mathrm{re}\right\rangle \subset \wW$$
We write $\wR_A(\lam)^+=\wR^+ \cap \wR_A(\lam)$ and $\wR(\lam)=\wR_\C(\lam)$ in case $\tau:\wS \rightarrow \wS /\wh \wS\cong \C$ is the quotient map and similarly $\wW(\lam)=\wW_\C(\lam)$.

\subsection{The critical level}

For $\lam \in \wh^*$ we define the \textit{level} of $\lam$ to be the complex number $\lam(c) \in \C$. If $\lam \sim_\K \mu$, we have $\lam(c)=\mu(c)$. Therefore, the equivalence class $\Lam$ of $\lam$ has a well-defined level and we have $\nu(c)= (\nu|\delta)$ for all $\nu \in \wh^*$.

\begin{lemma}[\cite{2}, Lemma 4.2]
For $\Lam \in \wh^*/\sim_\K$ the following are equivalent.
\begin{enumerate}
\item We have $\lam(c)= -\rho(c)$ for some $\lam \in \Lam$.
\item We have $\lam(c)= -\rho(c)$ for all $\lam \in \Lam$.
\item We have $\lam + \delta \in \Lam$ for all $\lam \in \Lam$.
\item We have $n \delta \in \wR_\K(\lam)$ for all $n \neq0$ and $\lam \in \Lam$.
\end{enumerate}
\end{lemma}
We call $\mathrm{crit}:= -\rho(c)$ the critical level.\\
Denote by $\wh^*_{\mathrm{crit}}$ the hyper plane which consists of all $\lam \in \wh^*$ with $\lam(c)=\mathrm{crit}$. Then for each $\lam \in \wh^*_{\mathrm{crit}}$ we have $(\lam + \rho| \delta)=0$.

\subsection{Restricted Verma modules}

Let $\lam \in \wh^*$ and $\tau: \wS \rightarrow A$ be a deformation algebra with a structure map that factors through the restriction map $\wS \rightarrow S$. This implies that $\tau(c)=\tau(d)=0$. We define $\del_A^-(\lam +\tau)$ to be the submodule of $\del_A(\lam +\tau)$ which is generated by the images of all homomorphisms $\del_A(\lam-n\delta +\tau) \rightarrow \del_A(\lam+\tau)$ for $n \in \N_{>0}$. Since $\tau(c)=\tau(d)=0$, we have $(\tau| \delta)=0$ and by Theorem \ref{kk} there is a map $\del_A(\lam-n\delta +\tau) \rightarrow \del_A(\lam+\tau)$ for every $n>0$, which is unequal to zero if and only if the level of $\lam$ is critical. If $\lam$ is non-critical, we get $\del_A^-(\lam +\tau)=\{0\}$.\\
We now define the \textit{restricted Verma module} as the quotient
$$\rdel_A(\lam +\tau)= \del_A(\lam+\tau)/\del_A^-(\lam+\tau)$$

\begin{remark}
Denote by $V^{\mathrm{crit}}(\g)$ the universal affine vertex algebra at the critical level and denote by $\mathfrak{z}$ its center. Then each smooth $[\wg,\wg]$-module $M$ carries the structure of a graded $\mathfrak{z}$-module. By a theorem of Feigin-Frenkel (cf. \cite{4}) $\mathfrak{z}$ yields an action on $M$ by the graded polynomial ring generated by infinitely many homogeneous elements 
$$\mathcal{Z}=\C[p_s^{(i)}\,|\,i=1,...,l, s \in \Z]=\bigoplus_{n \in \Z} \mathcal{Z}_n$$
Now a theorem of Frenkel and Gaitsgory (cf. \cite{5}) shows that for any critical weight $\lam \in \wh^*$ and $n<0$ there is a surjective map $\mathcal{Z}_n \rightarrow \Hom_{\wg}(\del(\lam+n\delta), \del(\lam))$. Thus, the restricted Verma module $\rdel(\lam)$ coincides with the quotient 
$$\del(\lam)^\mathrm{res}:=\del(\lam)/\sum_{\substack{n<0}}\mathcal{Z}_n \del(\lam).$$
\end{remark}

Let $\overline{\cdot}: \wh^* \twoheadrightarrow \h^*$ be the map induced by $\h \hookrightarrow \wh$. For any subset $\Lam \subset \wh^*$ we denote by $\overline{\Lam} \subset \h^*$ its image under $\overline{\cdot}$.

\begin{definition}
Let $\Lam \in \wh^*_\mathrm{crit}/\sim_\K$ be a critical equivalence class. We call $\Lam$
\begin{enumerate}
\item \textit{generic} if $\overline{\Lam} \subset \h^*$ contains exactly one element.
\item \textit{subgeneric} if $\overline{\Lam} \subset \h^*$ contains exactly two elements.
\end{enumerate}
\end{definition}

We call any weight contained in a generic (subgeneric, resp.) equivalence class a generic (subgeneric, resp.) weight. If $\Lam$ is subgeneric, there is a weight $\overline{\lam} \in \h^*$ and a finite root $\al \in R$ such that $\overline{\Lam}=\{\overline{\lam}, s_\al \cdot \overline{\lam}\}$.\\
Let $\lam \in \wh^*_\mathrm{crit}$ be a critical weight. Similarly to the integral roots of $\lam$ we now define the \textit{finite integral root system} (with respect to $\lam$ and the deformation algebra $A$) by
$$R_A(\lam):= \wR_A(\lam) \cap R=\{\al \in R\,|\, 2(\lam+\rho +\tau| \al)_A \in \Z(\al| \al)_A \}$$
and the \textit{finite integral Weyl group} by 
$$\W_A(\lam)=\wW_A(\lam) \cap \W$$
Again we write $R_A(\lam)^+=R^+\cap R_A(\lam)$ and $R(\lam)=R_\C(\lam)$ if the deformation algebra is $\C$.
For $\lam \in \wh^*_\mathrm{crit}$ and $\al \in R_A(\lam)$, such that $s_\al \cdot \lam \neq \lam$, we have either $s_\al \cdot \lam < \lam$ or $s_{-\al +\delta} \cdot \lam <\lam$. We define $\al \downarrow \lam$ to be the element in the set $\{s_\al \cdot \lam, s_{-\al+\delta}\cdot \lam\}$ which is smaller than $\lam$. Furthermore, we define inductively $\al \downarrow^n \lam:=\al \downarrow(\al \downarrow ^{n-1}\lam)$. In case $s_\al \cdot \lam=\lam$ we have $\al \downarrow \lam =\lam$.

\begin{thm}(\cite{2}, Corollary 4.10)\label{thmpa}
Let $\lam \in \wh^*_\mathrm{crit}$ and $\tau: \wS \rightarrow \K$ be a deformation algebra that is a field with structure map $\tau$, which factors over $S$.
\begin{enumerate}
\item If $\lam$ is generic, $\rdel_\K(\lam +\tau)$ is simple.
\item If $\lam$ is subgeneric with $R_\K(\lam)=\{\pm \al\}$, we have a short exact sequence
$$L_\K(\al \downarrow \lam) \hookrightarrow \rdel_\K(\lam) \twoheadrightarrow L_\K(\lam).$$
\end{enumerate}
\end{thm}
Note that the term subgeneric implies $\lam \neq \al \downarrow \lam$.

\section{The restricted Jantzen sum formula}\label{sec2}

In \cite{1}, chapter 6, the authors establish a Jantzen sum formula for baby Verma modules. It relates the sum of the characters of the Jantzen filtration to an alternating sum of characters of baby Verma modules with smaller highest weights. We deduce a similar formula for the restricted Verma modules at the critical level. We state the main result in 

\begin{thm}\label{Ja1}
Let $\lam \in \hc$. There is a filtration
$$\rdel(\lam)= \rdel(\lam)^0 \supset \rdel(\lam)^1 \supset \rdel(\lam)^2 \supset ...$$
with the properties
\begin{enumerate}
\item $\rdel(\lam)^1$ is the maximal submodule of $\rdel(\lam)$
\item $$\sum_{\substack{i>0}} \mathrm{ch} \rdel(\lam)^i =  \sum_{\substack{\alpha \in R(\lam)^{+}}} \left(\sum_{\substack{i>0}} \left(\mathrm{ch} \rdel(\alpha \downarrow^{2i-1}\lam)- \mathrm{ch} \rdel(\alpha \downarrow^{2i}\lam)\right)\right)$$
\end{enumerate}
\end{thm}

Note that the sum is taken over all finite, positive, integral roots $\al \in R(\lam)^+$.

\subsection{The Shapovalov determinant}

Let $\sigma: \wg \rightarrow \wg$ be a Chevalley-involution and define $\widehat{\mathfrak{n}}_+ := \bigoplus_{\al > 0} \wg_\al$ and $\widehat{\mathfrak{n}}_- := \bigoplus_{\al <0} \wg_\al$ where $\wg_\al$ is the root space of $\wg$ corresponding to the root $\al \in \wR$. Then we have a decomposition $U(\wg)=U(\wh)\oplus(\widehat{\mathfrak{n}}_-U(\wg)+U(\wg)\widehat{\mathfrak{n}}_+)$ and we denote by $\beta: U(\wg) \rightarrow S(\wh)$ the projection to the first summand of this decomposition.\\
The Shapovalov form is now defined as the bilinear pairing $F :U(\wg) \times U(\wg) \rightarrow S(\wh)$ with $F(x,y)= \beta(\sigma(x),y)$. It is symmetric and contravariant, i.e., for $u,x,y \in U(\wg)$ we have $F(\sigma(u)x,y)=F(x,uy)$.
For $\eta \in \N \wR^+$ we denote by $F_\eta$ the restriction of $F$ to the weight space $U(\widehat{\mathfrak{n}}_-)_{-\eta}$. Recall the isomorphism $\nu:\wh\stackrel{\sim}{\rightarrow} \wh^*$ induced by the bilinear form $(\cdot|\cdot)$ on $\wh$ and define $h_\al := \nu^{-1}(\al)$ for any root $\al \in \wR$.

\begin{thm}[\cite{6}, Theorem 1]\label{thmkk}
The determinant of 
$$F_\eta: U(\widehat{\mathfrak{n}}_-)_{-\eta} \times U(\widehat{\mathfrak{n}}_-)_{-\eta} \rightarrow S(\wh)$$ 
is, up to multiplication with a non-zero complex number, given by the formula
$$\mathrm{det} F_\eta = \prod_{\substack{\alpha \in \wR ^{+}}} \prod_{\substack{n=1}}^{\infty} \left( h_\al + \rho(h_\al) - n \frac{(\al|\al)}{2}\right) ^{\mathrm{mult}(\al) \cdot \mathcal{P}(\eta- n \alpha)}$$
where $\mathcal{P}$ is Kostant's partition function and $\mathrm{mult}(\al):= \mathrm{dim}_\C(\wg_\al)$.
\end{thm}

We equip the polynomial ring $\C[t]$ in one variable with two different structures of a deformation algebra. The first one is given by the map $\tau_1:\wS \twoheadrightarrow \C[t]$, where $\tau_1$ is induced by the inclusion of the line $\C\rho \subset \wh^*$. The second $\wS$-module structure $\tau_2:\wS \twoheadrightarrow \C[t]$ is given by the inclusion $\C\orho \subset \wh^*$. Recall that $\orho \in \h^*$ which implies that $\tau_2$ factors through the restriction map $\wS \twoheadrightarrow S$. For a more intuitional notation, we follow \cite{7} and define $\del_{\C[t]}(\lam +t \rho) := \del_{\C[t]}(\lam +\tau_1)$ and $\del_{\C[t]}(\lam +t\orho):=\del_{\C[t]}(\lam +\tau_2)$.\\
Note that for $\lam \in \wh^*_\mathrm{crit}$ critical and since $\tau_2(c)=\tau_2(d)=t\orho(c)=0$, we can construct the restricted Verma module $\rdel_{\C[t]}(\lam +t \orho)$. Let $\C(t)$ be the quotient field of $\C[t]$.

\begin{lemma}\label{lem1}
Let $\lam \in \wh^*_{\mathrm{crit}}$. Then $\rdel_{\C(t)}(\lam +t \orho)=\rdel_{\C[t]}(\lam +t \orho)\otimes_{\C[t]} \C(t)$ is simple.
\end{lemma}

\begin{proof}
If we proof that $R_{\C(t)}(\lam)=\emptyset$, the lemma follows from Theorem \ref{thmpa}. But since $(\orho|\al)\neq 0$ for all $\al \in R^+$ we get $2(\lam + \rho +t\orho|\al)_{\C(t)}\notin \Z(\al|\al)_{\C(t)}\subset \C(t)$ for all $\al \in R^+$.
\end{proof}

The Shapovalov form induces symmetric, contravariant bilinear forms on $\del_{\wS}(\lam + \epsilon')$ and $\rdel_S(\lam + \epsilon)$ where we denote by $\epsilon'\in \wh_{\wS}^*$ the canonical weight induced by $\wh \hookrightarrow \wS$ and by $\epsilon\in \wh_S^*$ its composition with $\wS \twoheadrightarrow S$. Moreover, it induces contravariant forms on all Verma modules $\del_{\C[t]}(\lam +t \rho)$, $\rdel_{\C[t]}(\lam +t \orho)$, $\del_{\C[t]}(\lam +t \orho)$, $\del(\lam)$ and $\rdel(\lam)$ we have to deal with in the rest of this paper. The contravariance of the forms implies for $\del(\lam)$ and $\rdel(\lam)$ that the radicals of the forms coincide with the maximal submodules of $\del(\lam)$ and $\rdel(\lam)$.

\subsection{Restricted Jantzen filtration}

Denote by $(\cdot,\cdot)$ the contravariant form on $\rdel_{\C[t]}(\lam + t\orho)$ induced by the Shapovalov form. We first define a filtration on $\rdel_{\C[t]}(\lam +t\orho)$ by 
$$\rdel_{\C[t]}(\lam+t\orho)^i:=\{m\in \rdel_{\C[t]}(\lam+t\orho)\,|\, ( m, \rdel_{\C[t]}(\lam+t\orho)) \subset t^i \C[t]\}$$
The \textit{Jantzen filtration} on $\rdel(\lam)$ is then defined by 
$$\rdel(\lam)^i:= \mathrm{im}(\rdel_{\C[t]}(\lam +t \orho)^i \hookrightarrow \rdel_{\C[t]}(\lam +t \orho) \twoheadrightarrow \rdel(\lam))$$
where the second map is specialization $t \mapsto 0$. In the same way we get the Jantzen filtration on $\del(\lam)$ as it is defined in \cite{6} using the deformed Verma module $\del_{\C[t]}(\lam +t \rho)$.

\begin{notation}
Let $\mu \leq \lam$. We denote the determinants of the contravariant bilinear forms on the $\mu$-weight spaces $\del_{\C[t]}(\lam+t \orho)_\mu$, $\del_{\C[t]}(\lam+t \rho)_\mu$ and $\rdel_{\C[t]}(\lam+t \orho)_\mu$ by $D_{\lam+t \orho}(\mu +t \orho)$, $D_{\lam+t \rho}(\mu +t \rho)$ and $\overline{D}_{\lam+t \orho}(\mu +t \orho)$.
\end{notation}

For a polynomial $P \in \C[t]$ denote by $\mathrm{ord}_t(P)$ the natural number $n\in \N$ with $t^n \mid P$ but $t^{n+1} \nmid P$. 

\begin{lemma}[\cite{7}, Lemma 5.1]\label{lemj}
For the Jantzen filtrations of the $\mu$-weight spaces of the non-restricted and restricted Verma modules we have the formulas
$$\sum_{\substack{i>0}} \mathrm{dim}_\C \Delta(\lam)^i_\mu=\mathrm{ord}_t (D_{\lam+t\rho}(\mu+t\rho))$$	
and
$$\sum_{\substack{i>0}} \mathrm{dim}_\C \rdel(\lam)^i_\mu=\mathrm{ord}_t (\overline{D}_{\lam+t\orho}(\mu+t\orho))$$	
\end{lemma}

We first want to describe the Jantzen filtration of a restricted Verma module with a highest weight, which is critical and subgeneric.

\begin{prop}\label{prop1}
Let $\lam \in \hc$ be subgeneric, i.e., $R(\lam)=\{\pm \al\}$ for a finite positive root $\al \in R^+$ and $\al\downarrow\lam \neq \lam$. Then the Jantzen filtration of $\rdel(\lam)$ is
$$\rdel(\lam)\supset L(\alpha \downarrow \lam) \supset 0$$
and we have the alternating sum formula
$$\sum_{\substack{i>0}} \mathrm{ch} \rdel(\lam)^i = \mathrm{ch} L(\alpha \downarrow \lam)= \mathrm{ch}\rdel(\alpha\downarrow \lam)-  \mathrm{ch} \rdel(\alpha\downarrow^2 \lam) +\mathrm{ch}\rdel(\alpha\downarrow^3 \lam)-...$$
\end{prop}

\begin{proof}
From Theorem \ref{thmpa} we see that $L(\al \downarrow \lam)$ is the maximal submodule of $\rdel(\lam)$. By the definition of the Jantzen filtration we conclude that $\rdel(\lam)^1$ coincides with the radical of the contravariant form on $\rdel(\lam)$ induced by the Shapovalov form. Thus, $\rdel(\lam)^1\cong L(\al\downarrow \lam)$. Thus, we are left to proof $\rdel(\lam)^2=\{0\}$.\\
For $\beta \in \wR^+$ and $n \in \N$ set $\mu_{\beta,n}:=\lam-\al\downarrow \lam- n \beta$. Then, by Theorem \ref{thmkk}, the polynomial $D_{\lam+t\orho}(\al \downarrow \lam +t \orho)\in \C[t]$ is, up to multiplication with a non-zero complex number, given by the product
$$\prod_{\substack{\beta \in \wR ^{+}}} \prod_{\substack{n=1}}^{\infty} \left( (\al\downarrow\lam)(h_\beta) + t\orho(h_\beta)+ \rho(h_\beta) - n\frac{(\beta|\beta)}{2} \right) ^{\mathrm{mult}(\beta) \cdot \mathcal{P}(\mu_{\beta,n})}$$
and similarliy $D_{\lam+t\rho}(\al \downarrow \lam +t \rho) \in \C[t]$ is given by
$$\prod_{\substack{\beta \in \wR ^{+}}} \prod_{\substack{n=1}}^{\infty} \left( (\al\downarrow\lam)(h_\beta) + t\rho(h_\beta)+ \rho(h_\beta) - n\frac{(\beta|\beta)}{2} \right) ^{\mathrm{mult}(\beta) \cdot \mathcal{P}(\mu_{\beta,n})}$$
Note that $\orho(h_\beta)=0$ for $\beta \in \wR^+$ if and only if $\beta=m \delta$ for $m \in \N \backslash \{0\}$. But for $\beta=m \delta$ with $m \in \N\backslash \{0\}$ we have $\PP(\lam-\al\downarrow\lam-n\beta)=0$ and we conclude
$$\mathrm{ord}_t(D_{\lam+t\rho}(\al \downarrow \lam +t \rho))= \mathrm{ord}_t(D_{\lam+t\orho}(\al \downarrow \lam +t \orho)).$$
But since $\al\downarrow \lam$ does not appear as a weight in the submodule $\del(\lam)^- \subset \del(\lam)$, we have
$$\overline{D}_{\lam+t\orho}(\al \downarrow \lam +t \orho)=D_{\lam+t\orho}(\al \downarrow \lam +t \orho)$$
Thus, by Lemma \ref{lemj}, we have
$$\sum_{\substack{i>0}}\mathrm{dim}_\C \rdel(\lam)^i_{\al \downarrow \lam}= \sum_{\substack{i>0}}\mathrm{dim}_\C \del(\lam)^i_{\al \downarrow \lam}$$
From the non-restricted Jantzen filtration of \cite{6}, one knows that the character of the simple module $L(\al\downarrow\lam)$ only appears once in the sum $\sum_{\substack{i>0}}\ch \del(\lam)^i$. Since $\al\downarrow\lam$ is a maximal weight of $\rdel(\lam)^1$ the character of $L(\al \downarrow \lam)$ can also only appear once in the sum $\sum_{\substack{i>0}}\mathrm{ch} \rdel(\lam)^i$ which implies $\rdel(\lam)^2 \neq L(\al \downarrow \lam)$ and thus $\rdel(\lam)^2=0$.\\
The second part of the proposition follows inductively from Theorem \ref{thmpa} since $\ch L(\al \downarrow \lam) = \ch \rdel(\al \downarrow \lam) -\ch L(\al \downarrow^2 \lam)$.
\end{proof}

\begin{remark}
One can find a slightly different proof of Proposition \ref{prop1} in \cite{9}, Lemma 4.9.
\end{remark}

Now, recall the canonical weight $\epsilon: \wh_S \rightarrow S$ and for $\mu \leq 0$ denote by $\overline{D}_{\epsilon}(\mu+\epsilon)$ the determinant of the contravariant form on $\rdel_S(\epsilon)_\mu$. Let $\phi:S\twoheadrightarrow \C[t]$ be the map induced by the restriction to $\lam+\C\orho \subset \wh^*$.\\
If $p \in S$ is a prime element and $a\in S$, we denote by $\mathrm{ord}_p(a)$ the integer $n \in \N$ such that $p^n \mid a$ but $p^{n+1} \nmid a$. By \cite{7}, chapter 5.6, we get for $a \in A$
\begin{equation}\label{eqn1}
\mathrm{ord}_t(\phi(a))=\sum_{\substack{p}}\mathrm{ord}_p(a) \mathrm{ord}_t(\phi(p))
\end{equation}

where $p$ runs over all classes of associated prime elements of $S$.\\
As in Lemma \ref{lem1} we see that, for $Q=Q(S)$ the quotient field of $S$, the restricted Verma module $\rdel_Q(\epsilon)\cong \rdel_S(\epsilon) \otimes_S Q$ is simple. We conclude that $\overline{D}_{\epsilon}(-\nu +\epsilon)\neq0$ for all $\nu \in \N \wR^+$ and $\phi(\overline{D}_{\epsilon}(-\nu +\epsilon))\neq0$ as well. Combining equation (\ref{eqn1}) with Lemma \ref{lemj}, we get
\begin{equation*}
\sum_{\substack{n>0}} \ch \rdel(\lam)^n = e^\lam \sum_{\substack{p}}\mathrm{ord}_t(\phi(p)) \sum_{\nu \in \N\wR^+}\mathrm{ord}_p(\overline{D}_{\epsilon}(-\nu +\epsilon))e^{-\nu}
\end{equation*}

We are now able to proof the general case. For any finite root $\al \in R$, $\langle \lam +\rho, \al^\vee\rangle-n$ differs from $h_\al +\rho(h_\al) -n\frac{(\al|\al)}{2}$ by multiplication with a non-zero complex number. For convenience, we work with the first polynomial in the following proof since we are only dealing with finite roots. We follow \cite{7}, chapter 5.7.
\begin{proof}[Proof of Theorem \ref{Ja1}]
If $\lam \in \hc$ and $\langle \lam +\rho, \al^\vee\rangle \notin \Z\backslash \{0\}$ for all $\al \in R^+$, then $\rdel(\lam)$ is simple, by Theorem \ref{thmpa}. Thus, the evaluation of the polynomial $\overline{D}_\epsilon(\epsilon -\nu) \in S$ at $\lam$ is unequal to zero. But this implies that all prime divisors of $\overline{D}_\epsilon(\epsilon -\nu)$ are of the form $\al^\vee+\rho(\al^\vee)-r$ where $\al \in R^+$ and $r\in \Z\backslash\{0\}$.\\
We define $\nu_{\al,r} \in \widehat{\Z[\wh^*]}$ by
$$\nu_{\al,r}= \sum_{\substack{\eta \in \N\wR^+}} \mathrm{ord}_{\al^\vee + \rho(\al^\vee)-r}(\overline{D}_\epsilon(\epsilon-\eta))e^{-\eta}.$$

For any $\lam \in \wh^*$ we have
$$\phi(\al^\vee +\rho(\al^\vee)-r)=\langle \lam +\rho,\al^\vee\rangle -r+t\langle\orho,\al^\vee\rangle \neq 0$$

If $r \neq \langle \lam +\rho,\al^\vee\rangle$, then $\mathrm{ord}_t(\langle \lam +\rho,\al^\vee\rangle-r+t\langle\orho,\al^\vee\rangle)=0$, but for $r=\langle \lam +\rho,\al^\vee\rangle$ we have $\mathrm{ord}_t(\langle \lam +\rho,\al^\vee\rangle-r+t\langle\orho,\al^\vee\rangle)=1$. Thus, $\al^\vee+\rho(\al^\vee)-\langle \lam +\rho,\al^\vee\rangle$ can only be a prime divisor of $\overline{D}_\epsilon(\epsilon-\eta)$ if $\al \in R(\lam)^+$. But then we conclude
\begin{equation}\label{eqn2}
\sum_{i>0}\ch \rdel(\lam)^i=\sum_{\substack{\al \in R(\lam)^+}} \nu_{\al,\langle \lam +\rho,\al^\vee\rangle}e^\lam
\end{equation}
To calculate the character $\nu_{\al,\langle \lam +\rho,\al^\vee\rangle}$ we find $\mu \in \hc$, such that $\langle \mu+\rho,\al^\vee\rangle=n=\langle \lam +\rho,\al^\vee\rangle$ and $\langle \mu+\rho,\beta^\vee\rangle \notin \Z$ for all $\beta \in R(\lam)^+\backslash \{\al\}$. Then we have that $\mu$ is subgeneric, since $R(\mu)=\{\pm \al\}$. In this case we already calculated the Jantzen filtration in Proposition \ref{prop1}:
$$\rdel(\mu) \supset L(\al\downarrow \mu) \supset 0$$
We conclude
$$\nu_{\al,n}e^\mu=\ch L(\al\downarrow \mu)= \sum_{\substack{i>0}} \left(\mathrm{ch} \rdel(\alpha \downarrow^{2i-1}\mu)- \mathrm{ch} \rdel(\alpha \downarrow^{2i}\mu)\right)$$

and because of the choice of $\mu$ we get
$$\nu_{\al,n}e^\lam= \sum_{\substack{i>0}} \left(\mathrm{ch} \rdel(\alpha \downarrow^{2i-1}\lam)- \mathrm{ch} \rdel(\alpha \downarrow^{2i}\lam)\right)$$
Thus, using equation (\ref{eqn2}) we get the result of Theorem \ref{Ja1}.
\end{proof}

As a consequence of Theorem \ref{Ja1} we get the linkage principle for restricted Verma modules at the critical level in the same way as in \cite{1}, chapter 6, or \cite{8}, Theorem 10.3. The linkage principle was already proved in \cite{3} introducing restricted projective objects in the restricted category $\OC$ over the Lie algebra $\wg$. Our proof, however, avoids the rather complicated construction of restricted projective objects.

\begin{cor}[\cite{3}, Theorem 5.1]
Let $\lam \in\hc$ and $\mu \in \wh^*$. Then $[\rdel(\lam):L(\mu)] \neq 0$ implies $\mu \in \wW(\lam)\cdot \lam$ and $\mu \leq \lam$.
\end{cor}

\begin{proof}
The statement is obvious for $\lam =\mu$ and it is also clear that $[\rdel(\lam):L(\mu)] \neq 0$ implies $\mu \leq \lam$. We use induction on $\lam -\mu$ and assume $\mu < \lam$. If $[\rdel(\lam):L(\mu)] \neq 0$, then also $[\rdel(\lam)^1:L(\mu)] \neq 0$ since $\rdel(\lam)^1 \subset \rdel(\lam)$ is the maximal submodule. But then the restricted Jantzen sum formula implies that $L(\mu)$ has to appear as a subquotient in some $\rdel(\al \downarrow^n \lam)$ where $\al \in R(\lam)^+$ and $n>0$. Our induction assumption then implies $\mu \in \wW(\al\downarrow^n\lam) \cdot (\al\downarrow^n\lam)$. But the definition of $\al \downarrow \lam$ implies $\wW(\al\downarrow^n\lam) \cdot (\al\downarrow^n\lam)= \wW(\lam) \cdot (\lam)$.
\end{proof}

\section{Acknowledgements}

I would like to thank Peter Fiebig for drawing my attention to the calculation of the above Jantzen sum formula. I am also very grateful for his patience for explaining me the underlying theory.

\bibliographystyle{amsplain}

\end{document}